\documentclass{amsart}

\usepackage{amssymb,amscd,amsmath,amsthm}
\usepackage{hyperref}
\usepackage{enumerate}
\usepackage{graphicx,tikz}
\usetikzlibrary{arrows, automata,chains,fit,shapes}

\newcommand\ra{\rightarrow}
\newcommand\e{\varepsilon}

\newcommand\lan{\mathcal{L}}
\newcommand\perms{\mathcal P}
\newcommand\diffrl{D_{\rho,\lambda}}
\newcommand\difflu{D_{\lambda,\mu}}

\newtheorem{thm}{Theorem}[section]
\newtheorem{prop}[thm]{Proposition}
\newtheorem{lem}[thm]{Lemma}

\newtheorem{defn}[thm]{Definition}

 \newcommand{\HU}{MR645539}
 \newcommand{\Estacks}{MR2240774}
\newcommand{\Knuth}{MR0445948}
\newcommand{\Bona}{MR1485138}
\newcommand{\Gessel}{MR1041448}
\newcommand{\SimSch}{MR829358}

\newcommand{\AALRa}{MR2176523}
\newcommand{\AALRb}{MR1453845}
\newcommand{\AALRc}{MR2000167}
\newcommand{\aerwz}{MR2199982}
\newcommand{\twopatterns}{MR2577304}

\newcommand{\Flajolet}{MR2483235}

\newcommand{\ElderCS}{ElderCS}
\newcommand{\wikipatterns}{wikipatterns}
\newcommand{\tony}{tony}
\newcommand{\PR}{PR}

\title[Permutations generated by a  depth 2 and  infinite stack]{Permutations generated by a  depth 2 and  infinite stack in series are algebraic}
\author[M. Elder]{Murray Elder}
\author[G. Lee]{Geoffrey  Lee}
\address{Mathematics,
University of Newcastle,
Callaghan NSW 2308, Australia}
\email{murray.elder@newcastle.edu.au}
\email{geoffrey.a.lee@uon.edu.au}
\author[A. Rechnitzer]{Andrew Rechnitzer}
\address{Department of Mathematics, University of British Columbia, Vancouver, British Columbia,  V6T-1Z2, Canada}
\email{andrewr@math.ubc.ca}


\date{\today}

\begin{document}

\begin{abstract}
We prove that the class of permutations generated by passing  an ordered sequence $12\dots n$ through a stack of depth 2 and an infinite stack in series 
is in bijection with an unambiguous context-free language, where a permutation of length $n$ is encoded by a string of length $3n$. 
 It follows that the sequence counting the number of permutations of each length has an algebraic generating function. We use the explicit context-free language to compute the generating function:
\begin{align*}
\sum_{n\geq 0} c_n t^n &= 
\frac{(1+q)\left(1+5q-q^2-q^3-(1-q)\sqrt{(1-q^2)(1-4q-q^2)}\right)}{8q}
\end{align*}
where $c_n$ is the number of permutations of length $n$ that can be generated, 
and $q \equiv q(t) = \frac{1-2t-\sqrt{1-4t}}{2t}$ is a simple variant of the 
Catalan generating function. This in turn implies that $c_n^{1/n} \to 
2+2\sqrt{5}$.
\end{abstract}
\maketitle

\section{Introduction}

 Let $p=p_1p_2\dots p_n$ and $q=q_1q_2\dots q_k$ be permutations 
of length $n\geq k$.
We say  $p$ {\em avoids} $q$ if there are no $k$ indices $i_1<\dots<i_k$ so that  for all $s,t$,
\[p_{i_s}<p_{i_t} \ \ \ \ \ \mathrm{if} \  \mathrm{and}\  \mathrm{only}\ \mathrm{if}  \ \ \ \ \ q_{s}<q_{t}.\]
For example, $25413$ avoids $123$ since it has no increasing subsequence of length 3.

Interest in sets of permutations that avoid a small set of ``patterns'' arose naturally in the study of stack-sorting (or equivalently stack-generating) algorithms.  Knuth showed that a permutation $p$ can be generated by passing the ordered sequence $12\dots |p|$ through an infinite  stack if and only if $p$ avoids $312$, and that permutations of length $n$ avoiding $312$ are counted by the Catalan numbers~\cite{\Knuth}. 






If $\mathfrak q$ is a list of permutations, let $Av_n(\mathfrak q)$ be the set of permutations of length $n$  that avoid  $q$ for each $q\in\mathfrak q$. 
We call $Av(\mathfrak q)=\bigcup_{n=0}^\infty Av_n(\mathfrak q)$ a {\em pattern-avoidance class}. A {\em basis} for a pattern avoidance class $Av(\mathfrak q)$ is a  set $\mathfrak p$ of pairwise avoiding permutations so that $Av(\mathfrak p)=Av(\mathfrak q)$. A class is {\em finitely based} if it is equal to $Av(\mathfrak p)$ for $\mathfrak p$ finite.
 The first author proved that the class of permutations generated by a stack of depth two and an infinite stack in series has a finite basis consisting of 20 permutations  \cite{\Estacks}.

The list of  pattern-avoidance classes  for which a generating function for the sequence counting  $Av_n(\mathfrak q)$ has been computed, or shown to be rational, algebraic or non-algebraic, is limited. 
Classes avoiding a single pattern of length $3$ are enumerated by the Catalan numbers \cite{\Knuth, \SimSch} and so have an  algebraic generating function. 
For length four, $Av(\{1342\})$ has an algebraic generating function \cite{\Bona}, $Av(\{1234\})$ has a generating function that is D-finite but not algebraic \cite{\Gessel}, and a closed form generating function for $Av(\{1324\})$ has not be found \cite{\aerwz, \tony}.
 It is known that  for any pattern $p$ of length four, $Av(\{p\})$ is in bijection with one of these three classes.
For single patterns of length greater than four, and classes avoiding two or more patterns, various isolated results are known \cite{\twopatterns, \wikipatterns}.



In this article we consider the class of permutations generated by passing an ordered sequence through a stack of depth 2 and infinite stack is series, which was shown to be finitely based by the first author \cite{\Estacks}.  The more general case of two infinite stacks in series has not been enumerated, and this work can be seen as a step towards this.  
Pierrot and   Rossin  recently proved a polynomial time algorithm to decide if a permutation can be sorted by two  stacks in series \cite{\PR}. The number of permutations sortable by 2 stacks in {\em  parallel} was recently solved by Albert and Bousquet-M\'elou~\cite{albert2014}. 

Several authors have considered the language-theoretic complexity of pattern avoidance classes --- see for example \cite{ \AALRc, \AALRa, \AALRb, \ElderCS}. 
Atkinson, Livesey, and Tulley \cite{\AALRb} showed that the set of permutations generated by passing an ordered sequence through a finite {\em token-passing network} is in bijection with a regular language. Initially we applied this technique to the finite network consisting of a stack of depth $2$ followed by a stack on depth $k$ in series, constructing a sequence of languages and  corresponding rational generating functions for small values of $k$. 
As $k$ increased, the rational generating functions appeared to converge to the algebraic function given in Theorem~\ref{thm:gfun} below.
However, his method does not constitute a proof. To prove the result we instead follow another path --- 
we establish a bijection between permutations generated and an unambiguous context-free language.
The generating function is then guarenteed to be algebraic by a well known theorem of Chomsky and Sch\"utzenberger. 

The main work in this article is to establish the bijection with the context-free language. It has been suggested that the method employed to transform the relatively simple pushdown-automaton description of the language to the quartic generating function should be much easier than the method we detail here. We would welcome any insights into this --- in our approach we merely apply the standard theory, and give the details for an interested reader (perhaps a student reading the paper).

\section{Acknowledgements}
The bulk of this paper is the result of a University of Newcastle summer vacation project
undertaken by the second author under supervision of the first. 
Research was supported by the Australian Research Council (ARC)  grant
FT110100178, and  the Natural Sciences and Engineering Research Council of Canada (NSERC).






\section{Establishing a bijection}\label{sec:bijection}

Let $\perms$ be the set of permutations that can be generated by a stack of depth 2 and infinite stack in series,
and fix $\rho,\lambda,\mu$ as the stack moves indicated in Figure~\ref{fig:rholambdamu}.


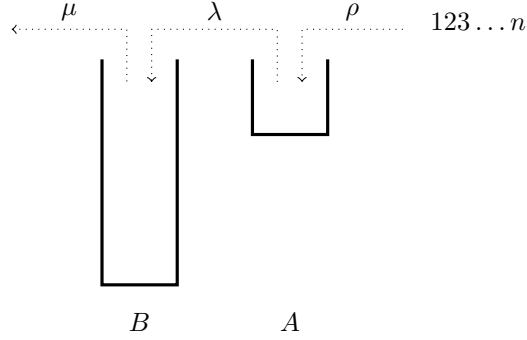
\begin{figure}[h!]
\begin{tikzpicture}
\draw[very thick] (0,2) -- (0,-1) -- (1,-1) -- (1,2);
\draw (2.5,-1.5) node {$A$};
\draw[very thick] (2,2) -- (2,1) -- (3,1) -- (3,2);
\draw (0.5,-1.5) node {$B$};
\draw[dotted, ->] (4,2.4) -- node[above] {$\rho$} (2.66,2.4) -- (2.66,1.7);

\draw (5,2.5) node {$123\dots n$};
\draw[dotted, bend left, ->] (2.33,1.7) -- (2.33,2.4) -- node[above] {$\lambda$} (0.66,2.4) -- (0.66,1.7);
\draw[dotted, <-] (-1.2,2.4) -- node[above] {$\mu$} (0.33,2.4) -- (0.33,1.7);


\end{tikzpicture}
\caption{Token passing moves $\rho,\lambda$ and $\mu$ for two stacks in series.}
\label{fig:rholambdamu}
\end{figure}

\begin{defn}[$D_{a,b}(u)$]
If $u$ is a word over an alphabet that includes the letters $a$ and $b$, 
define $D_{a,b}(u)$ to be the number of $a$ letters minus the number of $b$ letters contained in $u$.
\end{defn}

\begin{defn}[$\lan_{k,\infty}$]\label{defn:Lk-infty}
Let $k\in\mathbb N$. The language  $\lan_{k,\infty}$ is the set of words
 $w\in\{\rho,\lambda,\mu\}^*$ satisfying
\begin{enumerate}
\item $ \diffrl(u)\in [0,k]$ and $\difflu(u)\in [0,\infty)$ for all prefixes, $u$, of $w$, 
\item $\diffrl(w)=\difflu(w)=0$.
\end{enumerate}
\end{defn}

\begin{lem}\label{lem:L2inf}
A word  $w\in \{\rho,\lambda,\mu\}^*$ encodes a permutation in $\perms$ if and only if $w\in \lan_{2,\infty}$. Moreover, a word of length $3n$ in $\lan_{2,\infty}$ encodes a permutation of length $n$.
\end{lem}
\begin{proof}
The first claim is clear from the definition.
If $w\in \lan_{2,\infty}$ has $n$ $\rho$ letters, then  $\diffrl(w)=0$ implies $w$ has $n$ $\lambda$ letters, and $\difflu(w)=0$ then implies $w$ has $n$ $\mu$ letters, so the length of $w$ is $3n$.
\end{proof}

The language $ \lan_{2,\infty}$ consists of all possible ways to pass tokens through the system of stacks as in Figure~\ref{fig:rholambdamu}. We wish to find a sublanguage that is in bijection with $\perms$. From the set of all words in $ \lan_{2,\infty}$ that generate the same permutation, we will try to choose the string that outputs tokens as soon as possible, that is, has more $\mu$ letters closer to the front. The next definition will help to  formalise this.

\begin{defn}[$\mu$-ordering]
Define an ordering, $\prec_\mu$, on words in  $ \{\rho,\lambda,\mu\}^*$ as follows. Let $\theta: \{\rho,\lambda,\mu\}^*\ra \{\nu,\mu\}^*$ be a monoid homomorphism defined by $\theta(\mu)=\mu$ and $\theta(\rho)=\theta(\lambda)=\nu$.
 If $u\neq v$ as strings then $u\prec_\mu v$ if $|u|=|v|$ and $\theta(u)$ precedes $\theta(v)$ in  lexographic ordering on $\{\mu,\nu\}^*$ where $\mu<\nu$.
\end{defn}

For example, if $u=\rho\lambda\mu\rho\lambda\mu$ and     $v=\rho\lambda\rho\mu\lambda\mu$ then $u \prec_\mu  v$.
Note that both words generate the permutation $12$, and $u$ is obtained from $v$ by replacing the subword $\rho\mu$ by $\mu\rho$, which has no affect on the permutation being produced.
More generally we have the following.

\begin{lem}\label{lem:badsubwords}
Let $w\in \lan_{2,\infty}$.
\begin{enumerate}\item  If $w=w_0\rho\mu w_1$ then $w'=w_0\mu\rho w_1$ generates the same permutation as $w$, and $w'\prec_\mu w$.
\item  If   $w=w_0\rho\lambda w_1\lambda \mu w_2$ with  $\diffrl(w_0)=1$ and $w_1\in\lan_{1,\infty}$,
 then $w'=w_0\lambda\rho w_1\mu\lambda w_2$ generates the same permutation as $w$, and $w'\prec_\mu w$.
\item  If   $w= w_0\lambda\rho w_1\lambda \mu w_2$ with  $\diffrl(w_0)=1$ and $w_1\in\lan_{1,\infty}$,
 then $w'=w_0\rho\lambda w_1\mu\lambda w_2$ generates the same permutation as $w$, and $w'\prec_\mu w$.
 \end{enumerate}
 \end{lem}
\begin{proof}
In each case it is clear that $w'\prec_\mu w$.
We must show that in each case the two strings generate the same permutation. For case (1) this is clear since $\rho$ and $\mu$ do not interact.

For case (2),
 since  $\diffrl( w_0)=1$,
there must be one token (say $a$) left in the first stack after reading $w_0$, and since the next letter to be read is $\rho$, there must be one token (say $b$) ready to enter the first stack. See Figure~\ref{fig:replace}.

\begin{figure}
  \centering

\begin{tikzpicture}
\draw (-4,1) node {After $w_0$:};

\draw[very thick] (0,2) -- (0,-1) -- (1,-1) -- (1,2);
\draw (2.5,-1.5) node {$A$};
\draw[very thick] (2,2) -- (2,1) -- (3,1) -- (3,2);
\draw (0.5,-1.5) node {$B$};
\draw (2.5,1.35) node {$a$};
\draw (4.8,2.5) node {$b$ - - - -};
\draw (-1.6,2.5) node {- - - - };
\draw (0.5,-0.75) -- (0.5,-0.5);
\draw (0.5,-0.25) -- (0.5,0);
\end{tikzpicture}
  \\
  
\vspace{1cm}

  \begin{tikzpicture}
  \draw (-4,1) node {After $w_0\rho\lambda$:};
\draw[very thick] (0,2) -- (0,-1) -- (1,-1) -- (1,2);
\draw (2.5,-1.5) node {$A$};
\draw[very thick] (2,2) -- (2,1) -- (3,1) -- (3,2);
\draw (0.5,-1.5) node {$B$};
\draw (2.5,1.35) node {$a$};
\draw (4.8,2.5) node {- - - -};
\draw (-1.6,2.5) node { - - - -};
\draw (0.5,0.5) node {$b$};
\draw (0.5,-0.75) -- (0.5,-0.5);
\draw (0.5,-0.25) -- (0.5,0);
\end{tikzpicture}\\

\vspace{1cm}

    \begin{tikzpicture}
      \draw (-4,1) node {After $w_0\rho\lambda w_1\lambda\mu$:};
\draw[very thick] (0,2) -- (0,-1) -- (1,-1) -- (1,2);
\draw (2.5,-1.5) node {$A$};
\draw[very thick] (2,2) -- (2,1) -- (3,1) -- (3,2);
\draw (0.5,-1.5) node {$B$};

\draw (4.8,2.5) node {- - - -};
\draw (-1.6,2.5) node { - - - - $a$};
\draw (0.5,0.5) node {$b$};
\draw (0.5,-0.75) -- (0.5,-0.5);
\draw (0.5,-0.25) -- (0.5,0);
\end{tikzpicture}
  
    \caption{Stack configurations   in the proof of Lemma~\ref{lem:badsubwords}.}
  \label{fig:replace}
\end{figure}
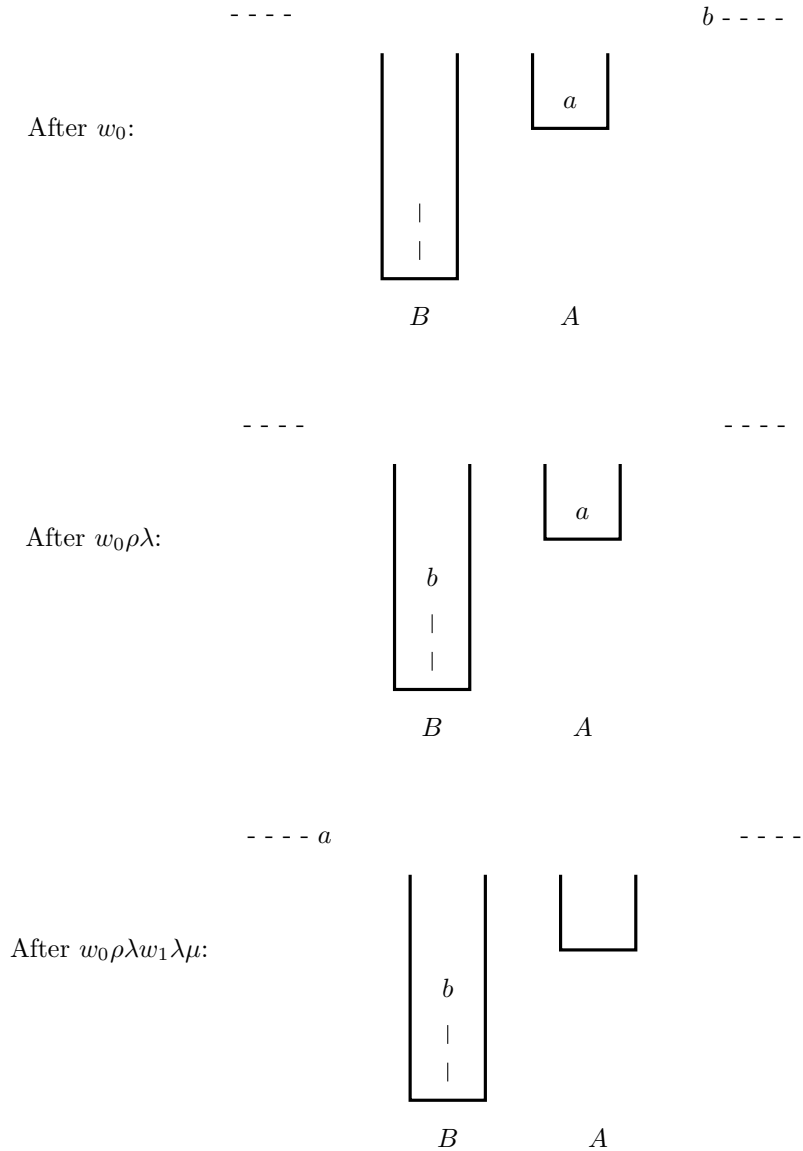

After reading $\rho\lambda$, $b$ moves to the top of stack $B$ and $a$ stays in stack $A$.
Reading $w_1$ leaves $a$ and $b$ in place and outputs some permutation of  input tokens after $b$.
Finally 
$\lambda\mu$ outputs $a$, leaving $b$ on the top of stack $B$ and stack $A$ empty.

Starting from the initial configuration in Figure~\ref{fig:replace}, the prefix $w_0\lambda\rho w_1\mu\lambda$ of $w'$ moves $a$ to the top of stack $B$ and places $b$ in stack $A$. The permutation generated by $w_1$ is then passed across as before, then $a$ is output, and finally $b$ is moved to stack $B$, leaving the stacks in the same configuration and the prefix of $w$.

A similar argument applies for Case (3) and is left to the reader.
\end{proof}

\begin{defn}[$\lan$]\label{DefL} The language $\lan$ is the set of words $w\in \lan_{2,\infty}$ that do not
\begin{enumerate}
\item contain $\rho\mu$,
\item  have a prefix $w_0\rho\lambda w_1\lambda \mu$ with $w_1\in\lan_{1,\infty}$ and $\diffrl(w_0)=1$,
\item  have a prefix $w_0\lambda\rho w_1\lambda \mu$ with $w_1\in\lan_{1,\infty}$ and $\diffrl(w_0)=1$.
\end{enumerate}
\end{defn}



\begin{lem}\label{lem:replaceKnuth}
Let $w\in \lan_{2,\infty}$. 
If either
\begin{enumerate}\item  $w=w_0 \rho\lambda w_1\lambda  w_2 \mu w_3$  with  $\diffrl(w_0)=1, w_1\in\lan_{1,\infty}$, and $w_2\in\lan_{2,\infty} $ generates a permutation that avoids $312$, or
\item   $w=w_0\lambda  \rho w_1\lambda  w_2 \mu w_3$  with  $\diffrl(w_0)=1, w_1\in\lan_{1,\infty}$, and $w_2\in\lan_{2,\infty} $ generates a permutation that avoids $312$, 
\end{enumerate}
then $w\not\in\lan$.
\end{lem}
\begin{proof}
Suppose for contradiction that $w\in\lan$,   $w=w_0 v w_1\lambda  w_2 \mu w_3$  with $v\in\{\rho\lambda, \lambda\rho\}$,  $\diffrl(w_0)=1, w_1\in\lan_{1,\infty}, w_2$ generates a permutation that avoids $312$, and moreover that $w_0$ is the longest prefix of $w$ with this property. That is, if $w=u_0v u_1\lambda  u_2 \mu w_3$  with $v\in\{\rho\lambda, \lambda\rho\},  \diffrl(u_0)=1, u_1\in\lan_{1,\infty}$ and $u_2$ generates a permutation that avoids $312$, then $|u_0|\leq |w_0|$.



Since $\diffrl(w_0vw_1)=1$ and $\lambda$ moves a token from stack $A$ to stack $B$, 
after reading $w_0vw_1\lambda$ we have no tokens in stack $A$, and some  token, say $a$, in stack $B$.  See Figure~\ref{fig:replaceKnuth}.
\begin{figure}
  \centering

\begin{tikzpicture}
\draw (-4,1) node {After $w_0v\lambda$:};

\draw[very thick] (0,2) -- (0,-1) -- (1,-1) -- (1,2);
\draw (2.5,-1.5) node {$A$};
\draw[very thick] (2,2) -- (2,1) -- (3,1) -- (3,2);
\draw (0.5,-1.5) node {$B$};
\draw (0.5,0.5)  node {$a$};
\draw (4.8,2.5) node {$b$ - - - -};
\draw (-1.6,2.5) node {- - - - };
\draw (0.5,-0.75) -- (0.5,-0.5);
\draw (0.5,-0.25) -- (0.5,0);
\end{tikzpicture}
\\

\vspace{1cm}

\begin{tikzpicture}
\draw (-4,1) node {After $w_0v\lambda \rho_b$:};

\draw[very thick] (0,2) -- (0,-1) -- (1,-1) -- (1,2);
\draw (2.5,-1.5) node {$A$};
\draw[very thick] (2,2) -- (2,1) -- (3,1) -- (3,2);
\draw (0.5,-1.5) node {$B$};
\draw (0.5,0.5)  node {$a$};
\draw (2.5,1.35) node {$b$};
\draw (4.8,2.5) node {- - - -};
\draw (-1.6,2.5) node {- - - - };
\draw (0.5,-0.75) -- (0.5,-0.5);
\draw (0.5,-0.25) -- (0.5,0);
\end{tikzpicture}
\\

\vspace{1cm}

\begin{tikzpicture}
\draw (-4,1) node {\begin{tabular}{l}
After $w_0v\lambda  \rho_b s\lambda_b$\\
 if $\difflu(s)>0$:\end{tabular}};

\draw[very thick] (0,2) -- (0,-1) -- (1,-1) -- (1,2);
\draw (2.5,-1.5) node {$A$};
\draw[very thick] (2,2) -- (2,1) -- (3,1) -- (3,2);
\draw (0.5,-1.5) node {$B$};
\draw (0.5,0.5)  node {$a$};
\draw (0.5,1.5)  node {$c$};
\draw (0.5,2)  node {$b$};
\draw (4.8,2.5) node { - - - -};
\draw (-1.6,2.5) node {- - - - };
\draw (0.5,-0.75) -- (0.5,-0.5);
\draw (0.5,-0.25) -- (0.5,0);

\draw (0.5,0.85) -- (0.5,1.1);
\end{tikzpicture}

    \caption{Stack configurations   in the proof of Lemma~\ref{lem:replaceKnuth}.}
  \label{fig:replaceKnuth}
\end{figure}
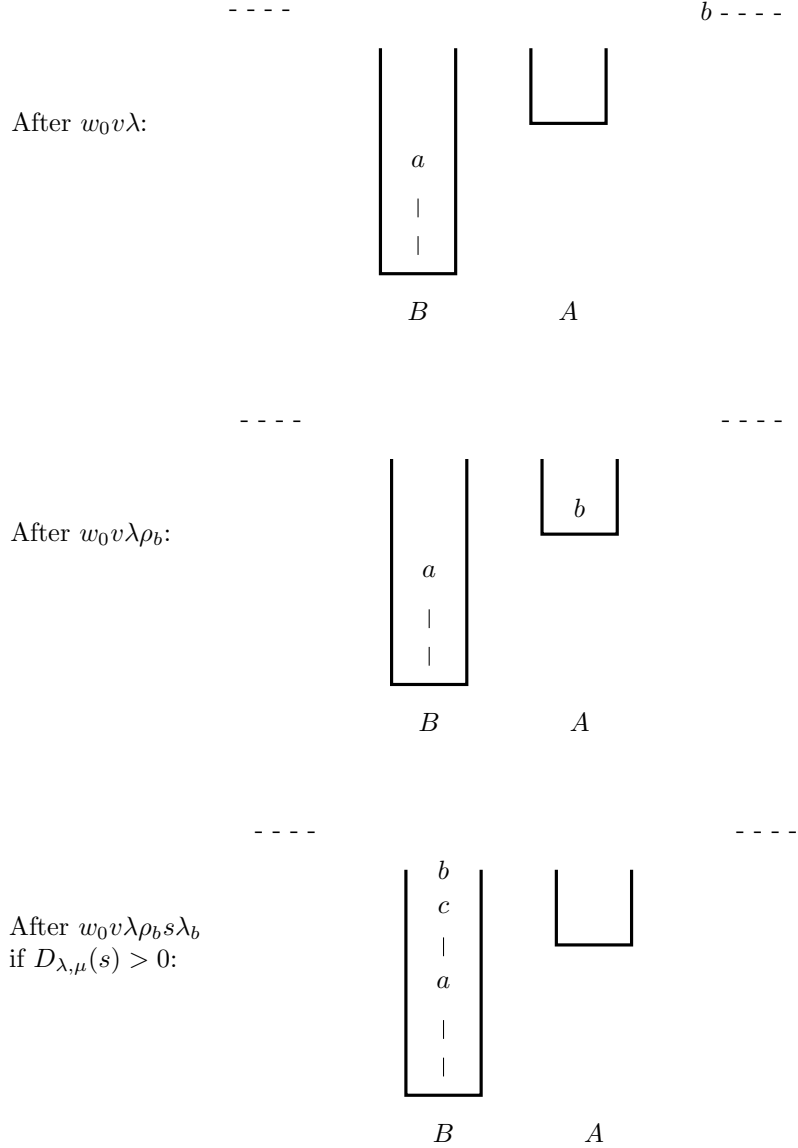

Since $w\in\lan$, $w_2$ cannot be empty, 
and since $w_2$ is a subword of $w\in\lan$ we have $w_2\in\lan$.
So $w_2$ moves some sequence of tokens completely through the stacks, leaving $a$ in place. 
The first letter of $w_2$ must be $\rho$, which moves some token, say $b$, onto stack $A$. Let $\rho_b,\lambda_b,\mu_b$ be the letters in $w_2$ that correspond to moving  $b$ through the stacks. Then $w_2$ has prefix $\rho_b s\lambda_bt\mu_b$ where $s,t$ are subwords.

Since stack $A$ contains $b$ while $s$ is read, if $\rho$ occurs in $s$ it must be immediately followed by  $\lambda$, so $\diffrl(u)\in[0,1]$ for all prefixes $u$ of $s$, and $\diffrl(s)=0$. Further, if  $\difflu(u)<0$ for any prefix $u$ of $s$, then  $a$ would be output. 
Either $\difflu(s)=0$ (and $s\in \lan$) or $\difflu(s)>0$.

If $s\in\lan_{1,\infty}$ then $t\in\lan_{2,\infty}$ and generates a permutation avoiding $312$ since it is a subword of $w_2$.
In this case
 $w$ has prefix $w=w_0 v w_1\lambda  \rho_b s\lambda_b t \mu_b$
with $\difflu(w_0 v w_1)=1$ and $t$ generating a permutation avoiding $312$, which contradicts the choice of $w_0$ as the longest such prefix.

Therefore we must have 
$\difflu(s)>0$. In this case, 
 after reading $s$
 at least one token, say $c$, remains on top of $a$ in stack $B$ when $b$ is moved into it. After reading $\lambda_b$, the stack configuration is as in the third diagram shown in Figure~\ref{fig:replaceKnuth}.

Note that $a<b<c$ since they are input in this order.
If $t\neq\e$ then it must contain at least one $\mu$ (it cannot leave a token covering $b$, and cannot just be $\rho$ or $\rho\rho$) so it moves a token $d>c$ to the output. This means $w_2$ generates the subpermutation $dbc$ which is order equivalent to $312$, contradicting our assumption.
Thus $t=\e$ and  $w_2$ has prefix $\rho_b s\lambda_b\mu_b$, with $s\in\{\rho\lambda, \mu\}^*$. 
Either $s$ ends with $\rho\lambda$, or $s=u\rho\lambda s'$ where $\difflu(u)=\difflu(s)$ since $\difflu$ starts at zero and increases to this value.
Thus $s'\in\lan$, and $w=w_0 v w_1\lambda  \rho_b u\rho\lambda s'  \lambda_b\mu_b$ with $\diffrl(w_0 v w_1\lambda  \rho_b u)=1$, which contradicts $w\in\lan$.
\end{proof}

\begin{thm}\label{thm:bijection}
There is a bijection between permutations in $\perms$ of length $n$ and words in $\lan$ of length $3n$.
\end{thm}
\begin{proof}
Consider 
the map that sends a word of length $3n$ in $\lan\subseteq \lan_{2,\infty}$ to the  permutation of length $n$ it generates.
If $\sigma\in\perms$ then there is some word $w\in \lan_{2,\infty}$ that generates it by Lemma~\ref{lem:L2inf}.
If $w\notin\lan$, then $w$ must  either  contain $\rho\mu$, or have prefix $w_0\rho\lambda w_1\lambda\mu$ or $w_0\lambda\rho w_1\lambda\mu$ with $\diffrl(w_0)=1$ and $w_1\in\lan_{1,\infty}$. We rewrite $w$ as follows.

While $w$ contains $\rho\mu$ or has prefix $w_0\rho\lambda w_1\lambda\mu$ or $w_0\lambda\rho w_1\lambda\mu$:
\begin{enumerate}
\item[1.] Replace $\rho\mu$ with $\mu\rho$
\item[2.] Replace $w_0\rho\lambda w_1\lambda \mu$ with $w_0\lambda\rho w_1\mu \lambda$
\item[3.] Replace $w_0\lambda\rho w_1\lambda \mu$ with $w_0\rho\lambda w_1\mu \lambda$
\end{enumerate}
Each iteration replaces the current word by a word which generates the same permutation and is shorter in the $\mu$-ordering by Lemma~\ref{lem:badsubwords}, so the procedure  must terminate (there are finitely many words less than $w$ in the $\mu$-ordering).
It follows that the map is surjective.  We complete the proof by showing it is injective.


Suppose we have two words $u,v\in\lan$ that generate the same permutation, and that $u\neq v$ as strings. Write \[u=u_1u_2\dots u_n \ \mathrm{and} \ v=v_1v_2\dots v_n\] where $u_i,v_i\in\{\rho,\lambda,\mu\}$.

Since $u,v\in\lan$ we have $u_1=v_1=\rho$.
Let $k\in [2,n]$ be such that $u_i=v_i$ for $i<k$ and $u_k\neq v_k$.
Let $z=u_1\dots u_{k-1}=v_1\dots v_{k-1}$, so \[u=zu_k\dots u_n \ \mathrm{and} \ v=zv_k\dots v_n.\]

First consider the case that one of $u_k,v_k$ is $\mu$. Without loss of generality assume $u=z\mu u_{k+1}\dots u_n$. Then $z$ must leave some token, say $a$, at the top of stack $B$, and $u_k=\mu$ outputs this token.

If $v_k=\lambda$, then  $a$ will be covered and $v$ will not be able to generate the same permutation. So we must have $v_{k}=\rho$. Then $v_{k+1}\neq \mu$. If $v_{k+1}=\lambda$ then $a$ is covered. So $v_{k+1}=\rho$.  Then  $v_{k+2}\neq \mu$, if $v_{k+2}=\lambda$ then $a$ is covered, and $v+{k+2}\neq \rho$ since stack $A$ contains two tokens. So we have a contradiction, 
and it follows that neither $u_k,v_k$ can be $\mu$.

Without loss of generality assume $u_k=\rho$ and $v_k=\lambda$.
Then $z$ must leave at least one token in stack $A$ to be followed by $\lambda$, and at most one token to be followed by $\rho$. Let $a$ be the token in $A$, and $b$ the token moved from the input by $u_k=\rho$. See Figure~\ref{fig:thm_I}. Note that we have $\diffrl(z)=1$.

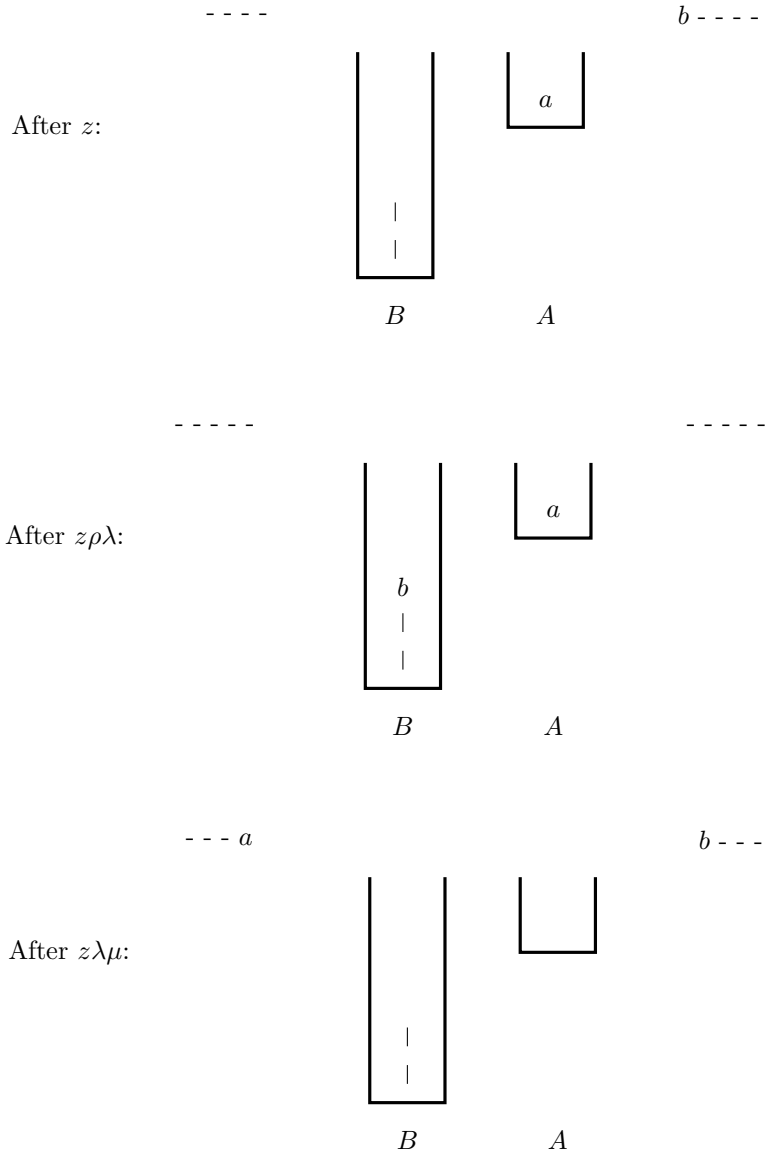
\begin{figure}[h!]
\begin{center}

\begin{tikzpicture}
\draw (-4,1) node 
{\begin{tabular}{l}
After $z$:\end{tabular}};

\draw[very thick] (0,2) -- (0,-1) -- (1,-1) -- (1,2);
\draw (2.5,-1.5) node {$A$};
\draw[very thick] (2,2) -- (2,1) -- (3,1) -- (3,2);
\draw (0.5,-1.5) node {$B$};
\draw (2.5,1.35) node {$a$};
\draw (4.8,2.5) node {$b$ - - - -};
\draw (-1.6,2.5) node {- - - - };
\draw (0.5,-0.75) -- (0.5,-0.5);
\draw (0.5,-0.25) -- (0.5,0);
\end{tikzpicture}
\\

\vspace{1cm}

\begin{tikzpicture}
\draw (-4,1) node 
{\begin{tabular}{l}
After $z\rho\lambda$:\end{tabular}};

\draw[very thick] (0,2) -- (0,-1) -- (1,-1) -- (1,2);
\draw (2.5,-1.5) node {$A$};
\draw[very thick] (2,2) -- (2,1) -- (3,1) -- (3,2);
\draw (0.5,-1.5) node {$B$};
\draw (4.8,2.5) node {- - - - -};
\draw (-2,2.5) node {- - - - -};
\draw (0.5,-0.75) -- (0.5,-0.5);
\draw (0.5,-0.25) -- (0.5,0);
\draw (2.5,1.35) node {$a$};
\draw (0.5,0.35) node{$b$};
\end{tikzpicture}
\\

\vspace{1cm}

\begin{tikzpicture}

\draw (-4,1) node 
{\begin{tabular}{l}
After $z\lambda\mu$:\end{tabular}};

\draw[very thick] (0,2) -- (0,-1) -- (1,-1) -- (1,2);
\draw (2.5,-1.5) node {$A$};
\draw[very thick] (2,2) -- (2,1) -- (3,1) -- (3,2);
\draw (0.5,-1.5) node {$B$};
\draw (4.8,2.5) node {$b$  - - -};
\draw (-2,2.5) node {- - - $a$};
\draw (0.5,-0.75) -- (0.5,-0.5);
\draw (0.5,-0.25) -- (0.5,0);
\end{tikzpicture}

\caption{\label{fig:thm_I} Stack configurations in Theorem~\ref{thm:bijection} where $u_k=\rho$ and $v_k=\lambda$.}
\end{center}
\end{figure}

In $u$, $z\rho$ must be followed by $\lambda$ since stack $A$ is full after the $\rho$ and $\rho$ cannot be followed by a $\mu$. So $u$ has prefix $z\rho\lambda$ and we have the configuration shown in the second diagram in Figure~\ref{fig:thm_I}.

In $v$, $z\lambda$ can be followed by either $\mu$ or $\rho$ but not $\lambda$ since stack $A$ is empty after  $v_k=\lambda$. 
Suppose $v_{k+1}=\mu$. Then after reading $z\lambda\mu$  we have the configuration shown in the third diagram in  Figure~\ref{fig:thm_I}.
Since $u$ and $v$ are assumed to produce the same permutation, the next $\mu$ letter appearing in $u$ after the prefix $z\rho\lambda$ must move $a$ to the output. Let $\lambda_a,\mu_a$ be the letters in $u$ that move the token $a$.
Then $u=z\rho\lambda u_1 \lambda_a u_2 \mu$ where $u_1,u_2\in\{\rho,\lambda\}^*$.
The subword $u_2$ cannot move tokens to cover $a$ in stack $B$, so cannot contain any $\lambda$ letters, and cannot contain any $\rho$ letters since it is followed by $\mu$, so it must be empty.
The subword $u_1$ must be of the form $(\rho\lambda)^i$ for $i\geq 0$, since it cannot move $a$. Then $u=z(\rho\lambda)^i\rho\lambda\lambda_a\mu_a$ with $\diffrl(z(\rho\lambda)^i)=1$, so $u\not\in\lan$.

It follows that $v_{k+1}=\rho$, so we have
\[u=z\rho\lambda \dots u_n, v=z\lambda\rho\dots v_n.\]
The two configurations of the stacks after reading the length $k+1$ prefixes of $u$ and $v$ respectively are shown in Figure~\ref{fig:cases}.

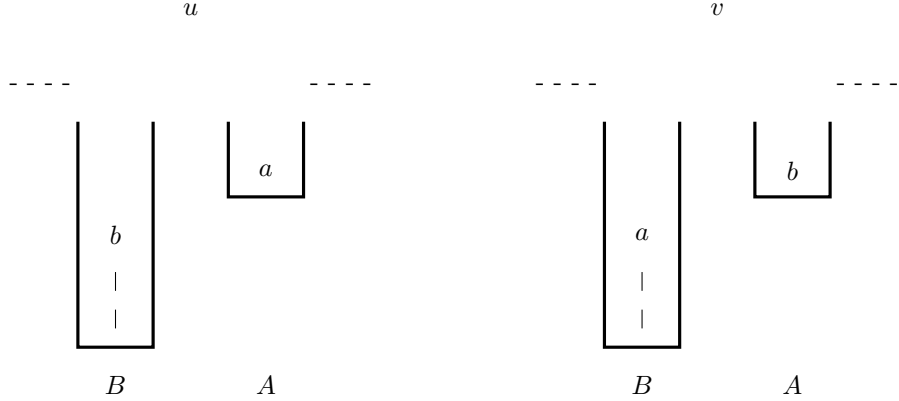
\begin{figure}[h!]
\begin{center}
\begin{tikzpicture}

\draw (-3.5,3.5) node {$u$};
\draw (3.5,3.5) node {$v$};
\draw[very thick] (-5,2) -- (-5,-1) -- (-4,-1) -- (-4,2);
\draw[very thick] (-3,2) -- (-3,1) -- (-2,1) -- (-2,2);
\draw[very thick] (2,2) -- (2,-1) -- (3,-1) -- (3,2);
\draw[very thick] (4,2) -- (4,1) -- (5,1) -- (5,2);

\draw (-4.5,-1.5) node {$B$};
\draw (-2.5,-1.5) node {$A$};
\draw (2.5,-1.5) node {$B$};
\draw (4.5,-1.5) node {$A$};

\draw (-5.5,2.5) node {- - - -};
\draw (-1.5,2.5) node {- - - -};
\draw (1.5,2.5) node {- - - -};
\draw (5.5,2.5) node {- - - -};

\draw (-4.5,-0.75) -- (-4.5,-0.5);
\draw (-4.5,-0.25) -- (-4.5,0);
\draw (-4.5, 0.5) node {$b$};
\draw (-2.5, 1.35) node {$a$};
\draw (2.5,-0.75) -- (2.5,-0.5);
\draw (2.5,-0.25) -- (2.5,0);
\draw (2.5, 0.5) node {$a$};
\draw (4.5, 1.35) node {$b$};
\end{tikzpicture}
\caption{\label{fig:cases}Stack configurations after $z\rho\lambda$ and $z\lambda\rho$  in Theorem~\ref{thm:bijection}.}
\end{center}
\end{figure}

We now consider two possibilities: either $a$ precedes $b$ in the permutation generated by $u$ and $v$, or $b$ precedes $a$.

\noindent \textbf{Case 1: $a$ precedes $b$}

 Mark the letters $\rho,\lambda,\mu$ in $u$ and $v$ that correspond to moving the token $a$, by appending the  subscript $a$.
So we have $u=z\rho\lambda w_1\lambda_a w_2\mu_a\dots u_n$ and $v=z\lambda_a\rho w \mu_a\dots v_n$ where $w,w_1, w_2\in\{\rho,\lambda,\mu\}^*$. 
 
 First consider the word $v$.
 Since $b$ must remain in stack $A$ until $a$ is output,   $w$ cannot end with $\rho$ and $w$ cannot leave any tokens covering $a$ in stack $B$, we have
$w\in\lan_{1,\infty}$. If $w$ is empty then $v$ contains $\rho\mu_a$ which means $v\not\in\lan$.
Thus $w$ is nonempty, so moves some tokens, say $t_1,\dots, t_s$, from the input to the output.

Since $u$ generates the same permutation as $v$, it must also move the tokens $t_1,\dots, t_s$ through the stacks and output them before $a$ is output. The subword $w_1$ cannot leave any tokens covering $a$ in stack $A$, so $w_1\in\{\rho\lambda, \mu\}^*$.


If $w_1$ leaves some tokens in stack $B$,
then these tokens must come after $t_s$ in the input, and so $w_1$ must feed all the tokens $t_1,\dots, t_s$ into the input, so $w_2$ cannot output any tokens, so cannot contain $\mu$, and cannot contain $\lambda$ since $a$ would be covered in stack $B$, and cannot be $\rho$ or $\rho\rho$ since it is followed by $\mu_a$, so $w_2$ is empty.
If $w_1$ ends with $\rho\lambda$, then write $w_1=p\rho\lambda$, and $z\rho\lambda w_1\lambda_a\mu_a=z\rho\lambda p \rho\lambda\lambda_a\mu_a$ with $\diffrl(z\rho\lambda p)=1$, so $u\not\in \lan$.
Otherwise $w_1$ ends in $\mu$. Since $w_1$ has more $(\rho\lambda)$ subwords than $\mu$ letters (it leaves tokens in stack $B$) then $w_1$ has some suffix $y\in\lan_{1,\infty}$ and prefix $p$ such that  $z=p\rho\lambda y$. So  we have  $z\rho\lambda w_1\lambda_a\mu_a=z\rho\lambda p \rho\lambda y \lambda_a\mu_a$ with $\diffrl(z\rho\lambda p)=1$ and $y\in\lan_{1,\infty}$ so $u\not\in \lan$.

Thus $w_1$ does not leave any tokens in stack $B$, so $w_1\in\lan_{1,\infty}$.  Let $t_1,\dots t_r$ with $r\leq s$ be the tokens moved to the output by $w_1$. 
The situtation is shown in Figure~\ref{fig:beforew2case1}. 

\begin{figure}[h!]
\begin{center}
\begin{tikzpicture}
\draw[very thick] (0,2) -- (0,-1) -- (1,-1) -- (1,2);
\draw (2.5,-1.5) node {$A$};
\draw[very thick] (2,2) -- (2,1) -- (3,1) -- (3,2);
\draw (0.5,-1.5) node {$B$};
\draw (4.5,2.5) node {$t_{r+1}\  \cdots\ t_s$ - -};
\draw (-1.5,2.5) node {- - - - -};
\draw (0.5,-0.75) -- (0.5,-0.5);
\draw (0.5,-0.25) -- (0.5,0);
\draw (0.5,0.75) node {$a$};
\draw (0.5,0.3) node{$b$};
\end{tikzpicture}
\caption{Stack configuration after $z\rho\lambda w_1\lambda_a$ in Case 1   in Theorem~\ref{thm:bijection}. }
\label{fig:beforew2case1}
\end{center}
\end{figure}
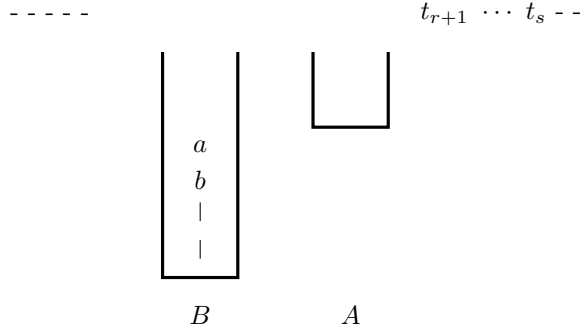

If $w_2$ is empty then $u$ has prefix $z\rho\lambda w_1\lambda_a\mu_a$ with $w_1\in\lan_{1,\infty}$ which is forbidden, so $w_2$ must move some tokens.
 The subword $w_2$ cannot leave any tokens in stack $B$.
Either $w_2$ leaves some tokens in stack $A$, or not.

 If $w_2$ leaves a token in stack $A$, this token cannot be one of $t_{r+1},\dots, t_s$ or else $v$ would generate a different permutation to $u$.
Therefore this token is moved into stack $A$ after $t_r$ by a letter $\rho$. This letter cannot be followed by $\mu$, and since it remains in stack $A$ it is not followed by $\lambda$. So this letter is either the last letter of $w_2$, or is followed by another $\rho$, which must also remain in stack $A$. Thus $w_2$ ends with $\rho$, but this is a contradiction since $w_2$ is followed by $\mu_a$.

 Thus $w_2$ does not leave any tokens in stacks $A$ or $B$, so  moves $t_{r+1},\dots, t_s$ from the input to the output, and $w_2\in\lan_{2,\infty}$. Note that $w_1w_2$ produces the same permutation of $t_1,\dots, t_s$ as $w$ does, and $w\in\lan_{1,\infty}$ so generates a $312$-avoiding permutation of $t_1,\dots,t_s$. The subword $w_1$ permutes the first $r$ tokens, and so $w_2$ must produce a permutation of $t_{r+1},\dots, t_s$ that avoids $312$. In this case $u$ has  prefix $z\rho\lambda w_1\lambda_a w_2 \mu_a$ where $\diffrl(z\rho\lambda)=1$, $w_1\in \lan_{1,\infty}$ and $w_2$ generates a $312$-avoider, so by Lemma~\ref{lem:replaceKnuth} $u$ must also contain a prefix that is not allowed if $u\in \lan$. This is a contradiction, so this case does not apply.


\noindent \textbf{Case  2: $b$ precedes $a$}

We return to the situation shown in Figure~\ref{fig:cases} with
 $u=z\rho\lambda\dots u_n$ and $v=z\lambda\rho\dots v_n$. Mark the letters $\rho,\lambda,\mu$ in $u$ and $v$ that correspond to moving the token $b$, by appending a subscript. 
 Then $u=z\rho_b\lambda_b w\mu_b\dots u_n$ and $v=z\lambda\rho_b w_1\lambda_b w_2\mu_b\dots v_n$ where $w,w_1, w_2\in\{\rho,\lambda,\mu\}^*$.
 
 First consider the word $u$.
 Since $a$ must remain in stack $A$ until $b$ is output,   $w$ cannot end with $\rho$ and $w$ cannot leave any tokens covering $b$ in stack $B$, we have
$w\in\lan_{1,\infty}$. If $w$ is empty then $u$ contains $\rho\mu_b$ which is forbidden, so $w$ moves some tokens, say $t_1,\dots, t_s$, from the input to the output.

Since $v$ generates the same permutation as $u$, it must also move the tokens $t_1,\dots, t_s$ through the stacks and output them before $b$ is output. 
The subword $w_1$ cannot leave any tokens covering $b$ in stack $A$, so $w_1\in\{\rho\lambda, \mu\}^*$.

If $w_1$ leaves some tokens in stack $B$,
then these tokens must appear after $t_s$ in the input, and so $w_1$ must feed the tokens $t_1,\dots, t_s$ into the input, so $w_2$ is empty (it cannot contain $\mu,\lambda$ and cannot end in $\rho$).
If $w_1$ ends with $\rho\lambda$, then write $w_1=p\rho\lambda$, and $z\lambda\rho_b w_1\lambda_b\mu_b=z\lambda\rho_b p \rho\lambda\lambda_b\mu_b$ with $\diffrl(z\lambda\rho_b p)=1$, so $v\not\in \lan$.
Otherwise $w_1$ ends in $\mu$. Since $w_1$ has more $(\rho\lambda)$ subwords than $\mu$ letters (it leaves tokens in stack $B$) then $w_1$ has some suffix $y\in\lan_{1,\infty}$ with $z=p\rho\lambda y$. So  we have  $z\lambda\rho_b w_1\lambda_b\mu_b=z\lambda\rho_b p \rho\lambda y \lambda_b\mu_b$ with $\diffrl(z\lambda \rho_b p)=1$ and $y\in\lan_{1,\infty}$ so $v\not\in \lan$.

Thus $w_1$ does not leave any tokens in stack $B$, so $w_1\in\lan_{1,\infty}$. 
Let $t_1,\dots t_r$ with $r\leq s$ be the tokens moved to the output by $w_1$. 
The situtation is shown in Figure~\ref{fig:beforew2}.

\begin{figure}[h!]
\begin{center}
\begin{tikzpicture}
\draw[very thick] (0,2) -- (0,-1) -- (1,-1) -- (1,2);
\draw (2.5,-1.5) node {$A$};
\draw[very thick] (2,2) -- (2,1) -- (3,1) -- (3,2);
\draw (0.5,-1.5) node {$B$};
\draw (4.5,2.5) node {$t_{r+1}\ t_2\ \cdots\ t_s$ - -};
\draw (-1.5,2.5) node {- - - - -};
\draw (0.5,-0.75) -- (0.5,-0.5);
\draw (0.5,-0.25) -- (0.5,0);
\draw (0.5,0.75) node {$b$};
\draw (0.5,0.3) node{$a$};
\end{tikzpicture}
\caption{Stack configuration after $z\lambda\rho_b w_1\lambda_b$ in Case 2   in Theorem~\ref{thm:bijection}. }
\label{fig:beforew2}
\end{center}
\end{figure}
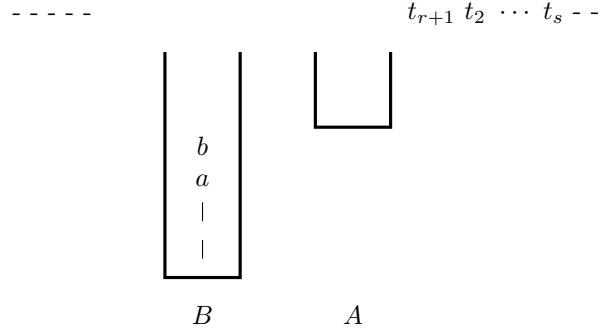

If $w_2$ is empty then $v$ has prefix $z\lambda\rho w_1\lambda_b\mu_b$ with $w_1\in\lan_{1,\infty}$ which is forbidden, so $w_2$ must move some tokens. 
 The subword $w_2$ cannot leave any tokens in stack $B$.
Either $w_2$ leaves some tokens in stack $A$, or not.

 If $w_2$ leaves a token in stack $A$, this token cannot be one of $t_{r+1},\dots, t_s$ or else $v$ would generate a different permutation to $u$.
Therefore this token is moved into stack $A$ after $t_r$ by a letter $\rho$. This letter cannot be followed by $\mu$, and since it remains in stack $A$ it is not followed by $\lambda$. So this letter is either the last letter of $w_2$, or is followed by another $\rho$, which must also remain in stack $A$. Thus $w_2$ ends with $\rho$, but this is a contradiction since $w_2$ is followed by $\mu_b$.

 Thus $w_2$ does not leave any tokens in stacks $A$ or $B$, so  moves $t_{r+1},\dots, t_s$ from the input to the output, and $w_2\in\lan_{2,\infty}$. Note that $w_1w_2$ produces the same permutation of $t_1,\dots, t_s$ as $w$ does, and $w\in\lan_{1,\infty}$ so generates a $312$-avoiding permutation of $t_1,\dots,t_s$. The subword $w_1$ permutes the first $r$ tokens, and so $w_2$ must produce a permutation of $t_{r+1},\dots, t_s$ that avoids $312$. In this case $v$ has  prefix $z\lambda\rho_b w_1\lambda_b w_2 \mu_b$ where $\diffrl(z\lambda\rho_b)=1$, $w_1\in \lan_{1,\infty}$ and $w_2$ generates a $312$-avoider, so by Lemma~\ref{lem:replaceKnuth} $v$ must also contain a prefix that is not allowed if $v\in \lan$. This is a contradiction, so we cannot have two such words $u$ and $v$.
\end{proof}

\subsection{A related class of permutations}

A natural question to ask is whether switching the order of the stacks makes any difference to the problem.
Let $\mathcal Q$ be the set of permutations that can be generated by passing an ordered sequence through an infinite  stack
followed by a  depth 2 stack in series. Each word $w\in\lan_{2,\infty}$ encodes a permutation in $\mathcal Q$ as follows: 
reading $w$ from right to left, for each $\mu$ move a token from the input to the infinite stack, for each $\lambda$ move a token from the infinite stack to the depth 2 stack, and 
for each $\rho$ move a token from the depth 2 stack to the output. 
It follows that $\mathcal P$ and $\mathcal Q$ are in bijection.




\section{Constructing a pushdown automaton}\label{sec:pda}

In this section we construct a deterministic pushdown automaton  accepting on empty stack, which accepts the language \[\lan\$=\{w\$ \mid w\in\lan\}.\]

A {\em pushdown automaton accepting on empty stack} $M$ is the following: \begin{enumerate}\item
$Q$ a finite set of {\em states},
\item $\Sigma$ a finite {\em input alphabet},
\item $\Gamma$ a finite {\em stack alphabet},
\item $q_0\in Q$ the {\em start state},
\item $0\in\Gamma$ a special stack symbol,
\item a map $\delta$  from $Q\times (\Sigma\cup\e)\times \Gamma$ to finite subsets of $Q\times (\Gamma^*)$,
\end{enumerate}
which runs as follows. Before reading input, the stack contains a single $0$.
Input strings are accepted as soon as the stack becomes empty.
A {\em configuration} of $M$  is a pair $(q, \omega)$ where $q$ is the current state and $\omega\in\Gamma^*$ is a string of stack symbols representing the contents of the stack (the first letter of $\omega$ is the top of the stack).
The notation $\delta(q_i,a, k)=\{(q_{j_1}, \gamma_1), \dots,(q_{j_s}, \gamma_s)\}$ means that 
if $M$ has the configuration $(q_i, k\omega)$ and $a\in\Sigma\cup\{\e\}$ is the next input letter to be read, then $M$ can move to the configuration $(q_{j_l}, \gamma_l\omega)$ for some $1\leq l\leq s$, removing the token $k$ from the top of the stack and replacing it by $\gamma_l$.

See \cite{\HU} for more details.

A pushdown automaton 
is {\em deterministic} if for each state $q$ and stack symbol $i$
\begin{enumerate}
\item if  $|\delta(q,\e, i)|=1$ then $|\delta(q,a,i)|=0$ for all $a\in \Sigma$,
\item for each $a\in\Sigma\cup\{\e\}$ the set
$\delta(q,a,i)$ has size at most one.
\end{enumerate}

Note  that a determistic pushdown automaton accepting on empty stack cannot accept the empty string (unless this is the only string it accepts) since there would have to be a transition $\delta(q_0, \e, 0)$ as well as a transition $\delta(q_0, a, 0)$ for some letter $a$.

 Let $M$ be the pushdown automaton shown in Figure~\ref{fig:PDA}, which accepts on empty stack.

\begin{figure}[h!]

\begin{tikzpicture}[node distance = 4cm,->,>=stealth']

\node [state, initial] (s) {$q_0$};
\node [state] (q1) [right of=s] {$q_1$};
\node [state] (q2) [right of=q1] {$q_2$};
\node [state] (q3) [below of=s] {$q_3$};
\node [state] (q4) [below of = q1] {$q_4$};
\node [state] (q5) [below of=q2] {$q_5$};

\node [state] (q7) [below of=q4] {$q_7$};

\node [state] (q6) [below of =q7] {$q_6$};
\node [state] (q8) [below  of = q5] {$q_8$};

\path (s) edge [loop, style={min distance = 1cm,in = 65, out = 115},above] node {$\begin{array}{l}\$,0\ra\e\\\mu,1\ra\e\\\mu,2\ra\e\end{array}$} (s);

\path (s) edge [above] node {$\rho, i\ra i$} (q1);
\path (q1) edge [above] node {$\rho, i\ra i$} (q2);
\path (q1) edge [left] node {$\lambda,i\ra 1i$} (q3);
\path (q2) edge [right] node {$\lambda,i\ra 1i$} (q5);
\path (q3) edge [style={in = 180, out = 270}, left] node {$\rho, i\ra i$} (q6);
\path (q5) edge [left] node {$\lambda,i\ra 1i$} (q7);
\path (q4) edge [right] node {$\rho, i\ra i$} (q2);
\path (q4) edge [above] node {$\lambda,i\ra 1i$} (q3);
\path (q5) edge [above] node {$\mu,1\ra\e$} (q4);
\path (q4) edge [loop, style={min distance = 1cm,in = 65, out = 115},above] node {$\begin{array}{l}\mu,1\ra\e\\\mu,2\ra\e\end{array}$} (q4);
\path (q3) edge [left] node {$\mu,1\ra\e$} (s);
\path (q8) edge  [style={min distance = 1cm,in = 285, out = 75},right] node {$\lambda,i\ra 2i$} (q5);
\path (q5) edge [style={min distance = 1cm,in = 105, out = 255},left] node {$\rho, i\ra i$} (q8);
\path (q5) edge [loop,style={min distance = 1cm,in = 335, out = 25},right] node {$\mu,2\ra\e$} (q5);

\path (q6) edge [style={min distance = 1cm,in =  255, out = 105}, left] node {$\lambda,i\ra 1i$} (q7);
\path (q7) edge [style={min distance = 1cm,in = 75 , out = 285 },right] node {$\rho, i\ra i$} (q6);

\path (q6) edge  [right] node {$\rho, i\ra i$} (q8);
\end{tikzpicture}

\caption{Pushdown automaton $M$ accepting on empty stack, with start configuration $(q_0,0)$. The  symbol $i\in\{0,1,2\}$ represents a stack token that is kept in place by a transition. \label{fig:PDA}}

\end{figure}
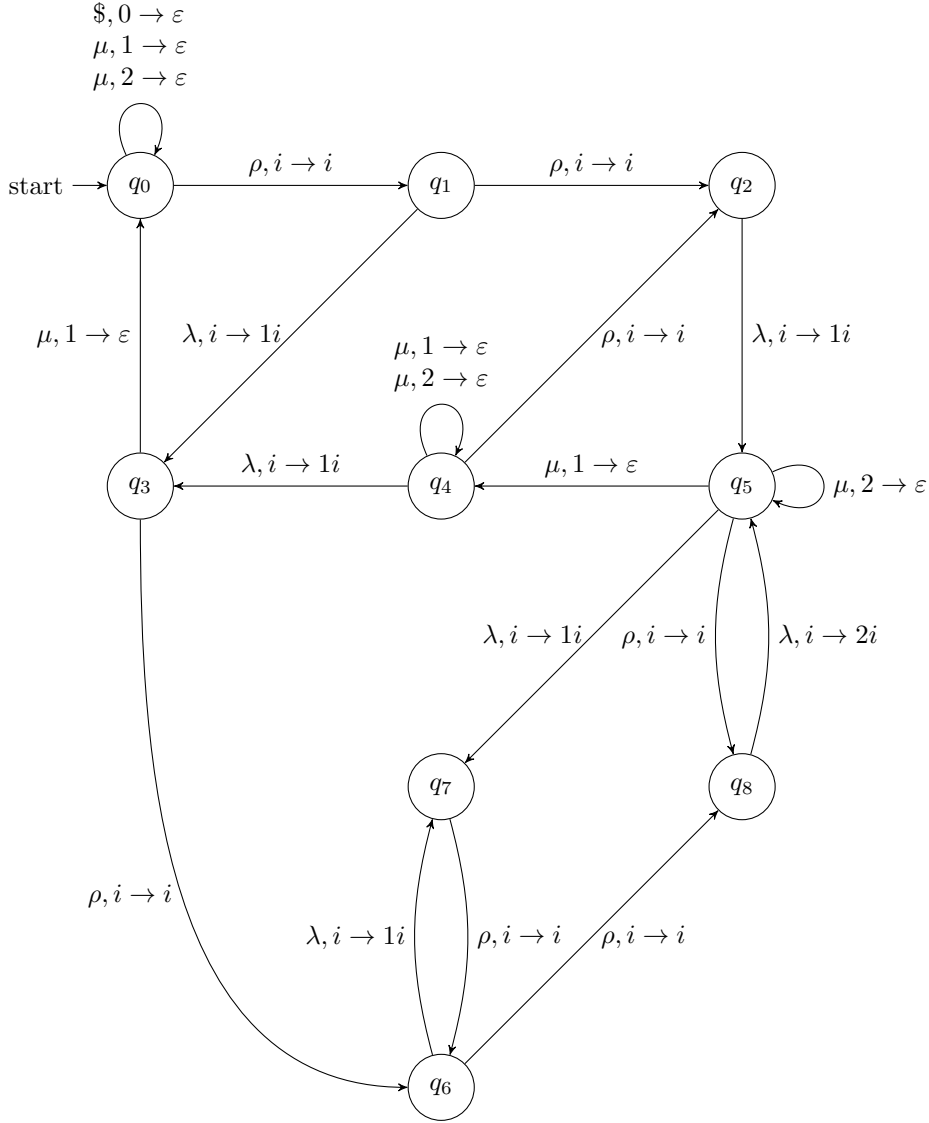


The pushdown automaton uses its stack to keep track of $\difflu$ as it reads its input, and its states to keep track of  $\diffrl$.  It uses the stack symbol $2$ as a device 
to flag when  the input has the potential to have a prefix of the form $w_0\rho\lambda$ or $w_0\lambda\rho$  with $\diffrl(w_0)=1$.
Paths $\rho\mu$ are forbidden. 
We will prove that the language of this automaton is precisely the language $\lan$.

Here is the formal description of $M$. Note that states $q_3,q_6,q_7$ are reached only when 1 is on top of the stack, and  $q_5,q_8$ are reached when  either 1 or 2 are on top of the stack, so we have omitted transitions from configurations that are not possible.

\begin{enumerate}\item 
 states $Q=\{q_0,\dots, q_{8}\}$, 
\item  input alphabet $\Sigma=\{\rho,\lambda, \mu, \$\}$, 
\item stack alphabet $\Gamma=\{0,1,2\}$,
\item start state $q_0$,
\item transition function $\delta$ defined as follows.
\[\begin{array}{lll}
\delta(q_0, \$,  0) = (q_0, \e) 	\\
\delta(q_0, \mu,  1) = (q_0, \e) 	\\
\delta(q_0, \mu,  2) = (q_0, \e)  	\\

\delta(q_3, \mu,  1) = (q_0, \e) \\
\delta(q_4, \mu,  1) = (q_4, \e) \\
\delta(q_4, \mu,  2) = (q_4, \e) \\

\delta(q_5, \mu,  1) = (q_4, \e)\\
\delta(q_5, \mu,  2) = (q_5, \e)\\
\\
\\
\\
\\
\\
\\

\end{array}\ \ \ \ \ \
\begin{array}{lll}
\delta(q_0, \rho,  0) = (q_1, 0) \\
\delta(q_0, \rho,  1) = (q_1, 1) 	\\
\delta(q_0, \rho,  2) = (q_1, 2) 	\\
	
\delta(q_1, \rho,  0) = (q_2, 0) \\
\delta(q_1, \rho,  1) = (q_2, 1) 	\\
\delta(q_1, \rho,  2) = (q_2, 2) 	\\

\delta(q_3, \rho,  1) = (q_6, 1)\\
\delta(q_4, \rho,  0) = (q_2, 0)\\
\delta(q_4, \rho,  1) = (q_2, 1) \\
\delta(q_4, \rho,  2) = (q_2, 2)\\

\delta(q_5, \rho,  1) = (q_8, 1) \\
\delta(q_5, \rho,  2) = (q_8, 2)\\

\delta(q_6, \rho,  1) = (q_8, 1)\\

\delta(q_7, \rho,  1) = (q_6, 1)\\
\end{array}\ \ \ \ \ \
\begin{array}{lll}
\delta(q_1, \lambda,  0) = (q_3, 10) 	\\
\delta(q_1, \lambda,  1) = (q_3, 11) 	\\
\delta(q_1, \lambda,  2) = (q_3, 12) 	\\

\delta(q_2, \lambda,  0) = (q_5, 10) 	\\
\delta(q_2, \lambda,  1) = (q_5, 11) 	\\
\delta(q_2, \lambda,  2) = (q_5, 12) 	\\

\delta(q_4, \lambda,  0) = (q_3, 10) 	\\
\delta(q_4, \lambda,  1) = (q_3, 11) 	\\
\delta(q_4, \lambda,  2) = (q_3, 12) 	\\

\delta(q_5, \lambda,  1) = (q_7, 11) 	\\
\delta(q_5, \lambda,  2) = (q_7, 12) 	\\
\delta(q_6, \lambda,  1) = (q_7, 11) 	\\

\delta(q_8, \lambda,  1) = (q_5, 21) 	\\
\delta(q_8, \lambda,  2) = (q_5, 22) 	\\

\end{array}\]
\end{enumerate}

 To prove that  $M$ accepts precisely the language $\lan$, we first show that $M$ is deterministic. This allows us to identify input words with unique paths in $M$ and simplify our arguments slightly.

\begin{lem}
The pushdown automaton $M$ is deterministic.
\end{lem}
\begin{proof}
The claim is easily verified by considering the formal description for $M$.
\end{proof}

\begin{prop}\label{prop:PDAcorrect}
The pushdown automaton $M$   accepts the language $\lan\$=\{w\$ \mid w\in\lan\}$.
\end{prop}
\begin{proof}
Since $M$ is deterministic, we identify input words with their corresponding unique path in $M$.

 Let $w\in\{\rho,\lambda,\mu\}^*$. We must show that
\begin{enumerate}
\item  if $w$ contains $\rho\mu$, then $w\$$ is rejected.
\item if $w$ fails to be in $\lan_{2,\infty}$, then $w\$$  is rejected, 
\item  if $w$  has a bad prefix (conditions (2) and (3) in Definition~\ref{DefL}), then $w\$$  rejected.
\item if $w\$$ is rejected, then $w\not\in \lan$.\end{enumerate}

 The only states that can be reached by a path $u\rho$ for $u\in\{\rho,\lambda, \mu\}^*$ from the start configuration
are  $q_1, q_2, q_6$ and $ q_8$ and 
since none are the source of a $\mu$ transition, any word containing $\rho\mu$ will be rejected.

Next, we show that if $w$ is not in $\lan_{2,\infty}$, then $w\$$ is rejected by $M$.  Each state represents the endpoint of a path labeling a prefix of an input string accepted by
the automaton. One can verify the values of $D_{\rho,\lambda}(u)$ for each path labeled $u$ ending at state $q_i$  given by Table~\ref{tableD}.

\begin{table}[h!]
\[
\begin{array}{|c|c|}
\hline
\mathrm{state} &  D_{\rho,\lambda}\\
\hline
q_0 & 0\\
q_1 & 1\\
q_2 & 2\\
q_3 & 0\\
q_4 & 1\\
q_5 & 1\\
q_6 & 1\\
q_7 & 0\\
q_8 & 2\\
\hline
\end{array}
\]
\caption{Value of $D_{\rho,\lambda}$ for any prefix ending at each state.\label{tableD}}
\end{table}


Let $h(u)$ be the height of the stack after reading  $u\in \{\rho,\lambda,\mu\}^*$ starting from the start configuration $(q_0,0)$. Then $h(\e)=1$, $h(u\rho)=h(u), h(u\lambda)=h(u)+1$ and $h(u\mu)=h(u)-1$ since $\lambda$  pushes a token to the stack, $\mu$ pops a token and $\rho$ keeps the stack unchanged. It follows that $h(u)=D_{\lambda,\mu}(u)+1$, and since $0$ stays on the stack until $\$$ is read, $h(u)\geq 1$ for all  prefixes $u\in \{\rho,\lambda,\mu\}^*$, so  $D_{\lambda,\mu}(u)\geq 0$.
If $w\$$ is accepted then the stack must contain only $0$ after reading $w$, so $D_{\lambda,\mu}(w)=0$.

It follows that if $D_{\rho,\lambda}>2, D_{\rho,\lambda},D_{\lambda,\mu}(u)<0$ for some prefix $u$, or $D_{\lambda,\mu}(w)\neq 0$, then $M$ will reject $w\$$.

Next, suppose $w\in \lan_{2,\infty}$ has no $\rho\mu$ substring and  a prefix of the form $w_0vw_1\lambda\mu$ where $D_{\rho,\lambda}(w_0)=1, v\in\{\rho\lambda,\lambda\rho\}$ and $w_1\in\lan_{1,\infty}$.  The string $w_0$ labels a path in the automaton starting at $q_0$ and ending at state $q_1, q_4, q_5$ or $q_6$ by Table~\ref{tableD}. From each of these states, reading $v=\rho\lambda$ ends in state $q_5$, and reading $v=\lambda\rho$ ends in state $q_6$.

From $q_5$, the  word $w_1$ labels a path that visits only states $q_5$ and $q_8$, since $D_{\lambda,\mu}(z)\geq 0$ for all prefixes $z$ of $w_1$, so the 1 on top of the stack before reading $w_1$ remains (and is covered by 2s, which are removed by the $\mu$ loop at $q_5$), and ends at $q_5$ since  $D_{\lambda,\mu}(w_1)= 0$.
From here reading $\lambda\mu$ is rejected.

From $q_6$, if $w_1=\e$ then 
$u\lambda\rho\lambda\mu$ is rejected. Otherwise 
 $w_1$ labels a path from $q_6$ to $q_8$ and then moves between $q_5$ and $q_8$, and ends at $q_5$.
 From here reading $\lambda\mu$ is rejected.


We have now established that if $w\not\in \lan$ then $w\$$ is rejected by $M$. To complete the proof we must show that if $w\$$ is rejected, then $w\not \in\lan$.
To show this, assume  $w\in \lan_{2,\infty}$ with no $\rho\mu$ substring, but $w\$$ is rejected by $M$. We will prove that $w$ must have a bad prefix.


Let $p$ be the longest prefix of $w\$$ labeling a path that is not rejected by $M$.
Since  $w\in \lan_{2,\infty}$
we have $D_{\lambda,\mu}(w)=0$, so if $w=p$, after reading $w$ the stack contains just $0$ so $w\$$ will be accepted, a contradiction.
Thus $p$ is strictly shorter than $w$. Let $w=pxw'$ where $x\in\{\rho\lambda,\mu\}$ is the next letter input after reading $p$.

We now consider the possible states where $p$ can end.

\begin{enumerate}
\item Suppose $p$ ends at $q_0$. Then $D_{\rho,\lambda}(p)=0$ so  $x\neq \lambda$. 
If the top of stack is $0$ then $D_{\lambda, \mu}(p)=0$ so $x\neq \mu$. 
Otherwise $M$ cannot reject on reading $\rho,\mu$.
\item 
Suppose $p$ ends at $q_1$, so its last letter is $\rho$,  and $D_{\rho,\lambda}(p)=1$.  Then $x\neq \mu$. Otherwise $M$ cannot reject on reading $\rho,\lambda$.
\item Suppose $p$ ends at $q_2$, so its last letter is $\rho$, and $D_{\rho,\lambda}(p)=2$. Then $x\neq \mu, \rho$. Otherwise $M$ cannot reject on reading $\lambda$.
\item Suppose $p$ ends at $q_3$, so  $D_{\rho,\lambda}(p)=0$ and the top of stack is $1$. Then $x\neq \lambda$. Otherwise $M$ cannot reject on reading $\rho,(\mu, 1\ra \e)$. 
\item Suppose $p$ ends at $q_4$, so $D_{\rho,\lambda}(p)=1$. The only way $M$ could reject is if the top of stack is $0$ and $x=\mu$, which is not possible since $w\in\lan_{2,\infty}$.
\item  Suppose $p$ ends at $q_5$, so $D_{\rho,\lambda}(p)=1$ and $1$ is on top of the stack. 
Then no letter will cause $M$ to reject.
\item  Suppose $p$ ends at $q_6$, so $D_{\rho,\lambda}(p)=1$ and $p$ ends with $\lambda\rho$. Then $x$ cannot be $\mu$, and otherwise $px$ is not rejected.

\item Suppose $p$ ends at $q_8$, so its last letter is $\rho$,  and $D_{\rho,\lambda}(p)=2$. Then $x\neq \mu,\rho$ and $M$ cannot reject if $x=\lambda$.
\end{enumerate}
These cases show that if $p$ ends at any state except $q_7$, then 
$M$ does not reject $w$ on reading the next input letter.
We finish the proof by showing that 
  if $p$ ends at $q_7$, then $px$ is a bad prefix.

Since $p$ ends at $q_7$,  $p$ ends with $\lambda$, $\diffrl(p)=2$, and $\difflu(p)>0$.
If $x=\rho$ then $px$ is not rejected. If $x=\lambda$ then $w\not\in\lan_{2,\infty}$. So we must have $x=\mu$.

Let $p=p_1\lambda$. If $p_1$ ends at $q_6$, then $p_1=p_2\lambda\rho$, and $px=p_2\lambda\rho\lambda\mu$ where $\diffrl(p_2)=1$ and so $px$ is a bad prefix.  The machine correctly rejects the string on reading $x=\mu$.

Otherwise $p_1$ ends at $q_5$. Either $p_1$ ends with $\rho\lambda$, or $\mu$. If $p_1=p_2\rho\lambda$ then $\diffrl(p_2)=1$ and $px=p_2\rho\lambda\lambda\mu$ is a bad prefix.  Otherwise $p_1$ ends in $\mu$, and must pop a token $2$ from the stack.  Let $\lambda_*$ be the last $\lambda$ letter in $p_1$ that pushed a $1$ on top of the stack  (which must exist, since all paths to $q_5$ must cross such an edge).  Write $p_1=p_2\lambda_*p_3\mu$.  

The letter $\lambda_*$ labels one of the following four edges: 
\begin{enumerate}
\item from $q_2$ to $q_5$, 
\item from $q_1$ to $q_3$,
\item from $q_4$ to $q_3$,
\item from $q_5$ to $q_7$,
\item from $q_6$ to $q_7$.
\end{enumerate}

In the first case, $p_2$ ends at $q_2$ so must have the form $p_2=u\rho$ with $\diffrl(u)=1$. Then $p_3\mu$ labels a path that moves between states $q_5$ and $q_8$, reading $\rho\lambda$ and pushing a 2, or reading $\mu$ and popping a 2, so $p_3\mu\in\lan_{1,\infty}$.
It follows that $w$ has the bad prefix $u\rho\lambda_*(p_3\mu)\lambda\mu$, and so $M$ correctly rejects it.

In the other four cases we have that $\diffrl(p_2)=1$ since $p_2$ ends at state $q_1, q_4, q_5$ or $q_6$, 
$\lambda_*$ must be immediately followed by a letter $\rho$, and $p_2\lambda_*\rho$ ends at state $q_6$. Let $p_3=\rho p_4$. Then $p_4\mu$ labels a path that starts at $q_6$, goes to $q_8$, then moves between states $q_5$ and $q_8$, reading $\rho\lambda$ and pushing a 2, or reading $\mu$ and popping a 2.
So  $p_4\mu\in\lan_{1,\infty}$.
It follows that $w$ has the bad prefix $p_2\lambda_*\rho(p_4\mu)\lambda\mu$, and so $M$ correctly rejects it.
\end{proof}

%
%

\section{Obtaining the generating function}


\begin{thm}\label{thm:gfun}
The sequence counting the number of permutations of each length in $\mathcal P$ has an algebraic generating function:
\begin{align*}
\sum_{n\geq 0} c_n z^n &= \frac{(1+q)\left(1+5q-q^2-q^3-(1-q)\sqrt{(1-q^2)(1-4q-q^2)}\right)}{8q}
\end{align*}
where  $c_n$ is the number of permutations in $\mathcal P$ of length $n$,  and $q \equiv q(z) = \frac{1-2z-\sqrt{1-4z}}{2z}$. 
\end{thm}

\begin{proof}
We convert the  pushdown automaton given  in the previous section to an unambiguous context-free language, following the standard procedure as described in Hopcroft and Ullman  \cite{\HU}.  
Theorem~10.12 of Hopcroft and Ullman guarantees that the grammar obtained from a deterministic pushdown
 automaton accepting on empty stack is $LR(0)$ and hence unambiguous.

We then apply the  Chomsky and Sch\"utzenberger theorem, as outlined for example in  \cite{\Flajolet}  I.5.4, to obtain an algebraic generating function.
Since each step in this procedure is constructive, we can find the generating function explicitly. 

We start by converting the pushdown automaton to a grammar. See Theorem~5.4 \cite{\HU} for full details.

Define a grammar with nonterminals $S$ and $[q_i,j,q_k]=N_{i,j,k}$ for each pair of states $q_i,q_k$ and stack symbol $j$.  
 The nonterminal $[q_i,j,q_k]$ represents a path in the configuration space of the pushdown automaton starting at $q_i$ with $j$ on top of the stack and ending at some state $q_k$. The productions ``fill out" these paths with subpaths according to the transitions that are possible.

The production rules are then defined as follows:
\begin{enumerate}
\item for each state $q_i$ we have a production
$S\ra N_{00i}$,
\item for each transition $\delta(q_i,a, j)=\{(q_k, \e)\}$ with $a\in\{\$,\mu\}$, add a production $N_{ijk}=a$,
\item for each transition $\delta(q_i,\rho, j)=\{(q_k, l)\}$,  add productions $N_{ijx}=\rho N_{klx}$ for $0\leq x \leq 8$,
\item for each transition $\delta(q_i,\lambda, j)=\{(q_k, lm)\}$,  add  productions $N_{ijx}=\lambda N_{kly}N_{ymx}$ $0\leq x,y\leq 8$.
\end{enumerate}

This gives the following set of productions, where $0\leq x,y\leq 8$:
\[\begin{array}{lcl}
 N_{000}& \ra &\$\\
   N_{010}& \ra &\mu\\
   N_{020}& \ra &\mu 	\\

   N_{310}& \ra &\mu \\
   N_{414}& \ra &\mu \\
  N_{424}& \ra &\mu \\

   N_{514}& \ra &\mu \\
  N_{525}& \ra &\mu\\
\\
\\
\\
\\
\\
\\
\end{array}\ \ \ \ \ \
\begin{array}{lcl}

N_{00x}& \ra & \rho N_{10x} \\
 N_{01x}& \ra & \rho N_{11x}  \\
 N_{02x}& \ra & \rho N_{12x}    \\
 N_{10x}& \ra & \rho N_{20x}  \\
 N_{11x}& \ra & \rho N_{21x}   \\
 N_{12x}& \ra & \rho N_{22x} 	\\

N_{31x}& \ra & \rho N_{61x}  \\
 N_{40x}& \ra & \rho N_{20x}   \\
 N_{41x}& \ra & \rho N_{21x}   \\
 N_{42x}& \ra & \rho N_{22x}   \\

 N_{51x}& \ra & \rho N_{81x}   \\
 N_{52x}& \ra & \rho N_{82x}  \\
 N_{61x}& \ra & \rho N_{81x}  \\

 N_{71x}& \ra & \rho N_{61x}  \\
\end{array}\ \ \ \ \ \
\begin{array}{lcl}
     N_{10x}& \ra & \lambda N_{31y}N_{y0x} \\
     N_{11x}& \ra & \lambda N_{31y}N_{y1x} \\
     N_{12x}& \ra & \lambda N_{31y}N_{y2x} 	\\
	
     N_{20x}& \ra & \lambda N_{51y}N_{y0x} 	\\
      N_{21x}& \ra & \lambda N_{51y}N_{y1x} 	\\
     N_{22x}& \ra & \lambda N_{51y}N_{y2x} 	\\

     N_{40x}& \ra & \lambda N_{31y}N_{y0x} 	\\
      N_{41x}& \ra & \lambda N_{31y}N_{y1x} 	\\
     N_{42x}& \ra & \lambda N_{31y}N_{y2x} 	\\
     N_{51x}& \ra & \lambda N_{71y}N_{y1x} 	\\
      N_{52x}& \ra & \lambda N_{71y}N_{y2x} 	\\
     N_{61x}& \ra & \lambda N_{71y}N_{y1x} 	\\
     N_{81x}& \ra & \lambda N_{52y}N_{y1x} 	\\
     N_{82x}& \ra & \lambda N_{52y}N_{y2x} 	\\
\end{array}\]

We can reduce the size of the grammar description  as follows.
First, observe that the only productions that eliminate nonterminals  (by generating $\$$ or $\mu$) are of the form $N_{\ast j k}$ for $k\in\{0,4,5\}$, and $j=0$ implies $k=0$.
Since  all productions with nonterminals on the right side have the form $N_{\ast ij}\ra \rho N_{\ast ij}$ or $N_{\ast ij}\ra \lambda N_{\ast\ast\ast}N_{\ast ij}$,
it follows that any nonterminal 
$N_{\ast\ast k}$ with $k$ not equal to  $0,4$ or $5$ cannot be eliminated, so we can exclude them from the grammar.

Also, if we start a derivation with $S\ra N_{00k}$ for $k\neq 0$, there will always be a nonterminal of the form $N_{\ast 0k}$ that cannot be eliminated. Therefore it suffices to make $N_{000}$  the start nonterminal and remove all productions involving $S$.

Lastly, the resulting grammar  contain nonterminals $N_{500},N_{504},N_{505}$ that will never produce a string of only terminals, since the configuration $(q_5,0)$ is never realised (to reach $q_5$ the top of stack symbol is either $1$ or $2$.
We modify the above grammar one step further by removing any production involving these nonterminals.

Taking these factors into consideration,  and collecting productions with the same left side together we obtain the following grammar:

\bigskip

$
 \begin{array}{lll}
 N_{000}  & \ra & \$ \mid \rho N_{100},\\
      N_{004}  & \ra & \rho N_{104},\\
    N_{005}  & \ra & \rho N_{105},\\
    N_{010}  & \ra & \mu \mid \rho N_{110},\\
    N_{014}  & \ra & \rho N_{114},\\
    N_{015}  & \ra & \rho N_{115},\\
    N_{020}  & \ra & \mu \mid \rho N_{120},\\
    N_{024}  & \ra & \rho N_{124},\\
    N_{025}  & \ra & \rho N_{125},
     \end{array}$
     
    $
 \begin{array}{lll}
    N_{100}  & \ra & \rho N_{200} \mid \lambda N_{310} N_{000} \mid \lambda N_{314} N_{400},\\
    N_{104}  & \ra & \rho N_{204} \mid \lambda N_{310} N_{004} \mid \lambda N_{314} N_{404},\\
    N_{105}  & \ra & \rho N_{205} \mid \lambda N_{310} N_{005} \mid \lambda N_{314} N_{405},\\
    N_{110}  & \ra & \rho N_{210} \mid \lambda N_{310} N_{010} \mid \lambda N_{314} N_{410} \mid \lambda N_{315} N_{510},\\
    N_{114}  & \ra & \rho N_{214} \mid \lambda N_{310} N_{014} \mid \lambda N_{314} N_{414} \mid \lambda N_{315} N_{514},\\
    N_{115}  & \ra & \rho N_{215} \mid \lambda N_{310} N_{015} \mid \lambda N_{314} N_{415} \mid \lambda N_{315} N_{515},\\
         \end{array}$
     
    $
 \begin{array}{lll}
    N_{120}  & \ra & \rho N_{220} \mid \lambda N_{310} N_{020} \mid \lambda N_{314} N_{420} \mid \lambda N_{315} N_{520},\\
    N_{124}  & \ra & \rho N_{224} \mid \lambda N_{310} N_{024} \mid \lambda N_{314} N_{424} \mid \lambda N_{315} N_{524},\\
    N_{125}  & \ra & \rho N_{225} \mid \lambda N_{310} N_{025} \mid \lambda N_{314} N_{425} \mid \lambda N_{315} N_{525},
         \end{array}$
     
    $
 \begin{array}{lll}
    N_{200}  & \ra & \lambda N_{510} N_{000} \mid \lambda N_{514} N_{400},\\
    N_{204}  & \ra & \lambda N_{510} N_{004} \mid \lambda N_{514} N_{404},\\
    N_{205}  & \ra & \lambda N_{510} N_{005} \mid \lambda N_{514} N_{405},\\
    N_{210}  & \ra & \lambda N_{510} N_{010} \mid \lambda N_{514} N_{410} \mid \lambda N_{515} N_{510},\\
    N_{214}  & \ra & \lambda N_{510} N_{014} \mid \lambda N_{514} N_{414} \mid \lambda N_{515} N_{514},\\
    N_{215}  & \ra & \lambda N_{510} N_{015} \mid \lambda N_{514} N_{415} \mid \lambda N_{515} N_{515},\\
    N_{220}  & \ra & \lambda N_{510} N_{020} \mid \lambda N_{514} N_{420} \mid \lambda N_{515} N_{520},\\
    N_{224}  & \ra & \lambda N_{510} N_{024} \mid \lambda N_{514} N_{424} \mid \lambda N_{515} N_{524},\\
    N_{225}  & \ra & \lambda N_{510} N_{025} \mid \lambda N_{514} N_{425} \mid \lambda N_{515} N_{525},
         \end{array}$
     
    $
 \begin{array}{lll}
    N_{310}  & \ra & \mu \mid \rho N_{610},\\
    N_{314}  & \ra & \rho N_{614},\\
    N_{315}  & \ra & \rho N_{615},
         \end{array}$
     
    $
 \begin{array}{lll}
    N_{400}  & \ra & \rho N_{200} \mid \lambda N_{310} N_{000} \mid \lambda N_{314} N_{400},\\
    N_{404}  & \ra & \rho N_{204} \mid \lambda N_{310} N_{004} \mid \lambda N_{314} N_{404},\\
    N_{405}  & \ra & \rho N_{205} \mid \lambda N_{310} N_{005} \mid \lambda N_{314} N_{405},\\
    N_{410}  & \ra & \rho N_{210} \mid \lambda N_{310} N_{010} \mid \lambda N_{314} N_{410} \mid \lambda N_{315} N_{510},\\
    N_{414}  & \ra & \mu \mid \rho N_{214} \mid \lambda N_{310} N_{014} \mid \lambda N_{314} N_{414} \mid \lambda N_{315} N_{514},\\
    N_{415}  & \ra & \rho N_{215} \mid \lambda N_{310} N_{015} \mid \lambda N_{314} N_{415} \mid \lambda N_{315} N_{515},\\
    N_{420}  & \ra & \rho N_{220} \mid \lambda N_{310} N_{020} \mid \lambda N_{314} N_{420} \mid \lambda N_{315} N_{520},\\
    N_{424}  & \ra & \mu \mid \rho N_{224} \mid \lambda N_{310} N_{024} \mid \lambda N_{314} N_{424} \mid \lambda N_{315} N_{524},\\
    N_{425}  & \ra & \rho N_{225} \mid \lambda N_{310} N_{025} \mid \lambda N_{314} N_{425} \mid \lambda N_{315} N_{525},
         \end{array}$
     
    $
 \begin{array}{lll}
    N_{510}  & \ra & \rho N_{810} \mid \lambda N_{710} N_{010} \mid \lambda N_{714} N_{410} \mid \lambda N_{715} N_{510},\\
    N_{514}  & \ra & \mu \mid \rho N_{814} \mid \lambda N_{710} N_{014} \mid \lambda N_{714} N_{414} \mid \lambda N_{715} N_{514},\\
    N_{515}  & \ra & \rho N_{815} \mid \lambda N_{710} N_{015} \mid \lambda N_{714} N_{415} \mid \lambda N_{715} N_{515},\\
    N_{520}  & \ra & \rho N_{820} \mid \lambda N_{710} N_{020} \mid \lambda N_{714} N_{420} \mid \lambda N_{715} N_{520},\\
    N_{524}  & \ra & \rho N_{824} \mid \lambda N_{710} N_{024} \mid \lambda N_{714} N_{424} \mid \lambda N_{715} N_{524},\\
    N_{525}  & \ra & \mu \mid \rho N_{825} \mid \lambda N_{710} N_{025} \mid \lambda N_{714} N_{425} \mid \lambda N_{715} N_{525},
         \end{array}$
     
    $
 \begin{array}{lll}
    N_{610}  & \ra & \rho N_{810} \mid \lambda N_{710} N_{010} \mid \lambda N_{714} N_{410} \mid \lambda N_{715} N_{510},\\
    N_{614}  & \ra & \rho N_{814} \mid \lambda N_{710} N_{014} \mid \lambda N_{714} N_{414} \mid \lambda N_{715} N_{514},\\
    N_{615}  & \ra & \rho N_{815} \mid \lambda N_{710} N_{015} \mid \lambda N_{714} N_{415} \mid \lambda N_{715} N_{515},
             \end{array}$
     
    $
 \begin{array}{lll}
    N_{710}  & \ra & \rho N_{610},\\
    N_{714}  & \ra & \rho N_{614},\\
    N_{715}  & \ra & \rho N_{615},
             \end{array}$
     
    $
 \begin{array}{lll}
    N_{810}  & \ra & \lambda N_{520} N_{010} \mid \lambda N_{524} N_{410} \mid \lambda N_{525} N_{510},\\
    N_{814}  & \ra & \lambda N_{520} N_{014} \mid \lambda N_{524} N_{414} \mid \lambda N_{525} N_{514},\\
    N_{815}  & \ra & \lambda N_{520} N_{015} \mid \lambda N_{524} N_{415} \mid \lambda N_{525} N_{515},\\
    N_{820}  & \ra & \lambda N_{520} N_{020} \mid \lambda N_{524} N_{420} \mid \lambda N_{525} N_{520},\\
    N_{824}  & \ra & \lambda N_{520} N_{024} \mid \lambda N_{524} N_{424} \mid \lambda N_{525} N_{524},\\
    N_{825}  & \ra & \lambda N_{520} N_{025} \mid \lambda N_{524} N_{425} \mid \lambda N_{525} N_{525}.
 \end{array}$

\bigskip

The next step is to convert nonterminals to generating functions, terminals to $z$ and productions to equations, as described in 
\cite{\Flajolet} I.5.4.

\bigskip

$
 \begin{array}{lll}
 f_{000}  & = & z + z f_{100},\\
      f_{004}  & = & z f_{104},\\
    f_{005}  & = & z f_{105},\\

         \end{array}$
     
    $
 \begin{array}{lll}
     f_{010}  & = & z + z f_{110},\\
    f_{014}  & = & z f_{114},\\
    f_{015}  & = & z f_{115},\\
    f_{020}  & = & z + z f_{120},\\
    f_{024}  & = & z f_{124},\\
    f_{025}  & = & z f_{125},
         \end{array}$
     
    $
 \begin{array}{lll}
    f_{100}  & = & z f_{200} + z f_{310} f_{000} + z f_{314} f_{400},\\
    f_{104}  & = & z f_{204} + z f_{310} f_{004} + z f_{314} f_{404},\\
    f_{105}  & = & z f_{205} + z f_{310} f_{005} + z f_{314} f_{405},\\
    f_{110}  & = & z f_{210} + z f_{310} f_{010} + z f_{314} f_{410} + z f_{315} f_{510},\\
    f_{114}  & = & z f_{214} + z f_{310} f_{014} + z f_{314} f_{414} + z f_{315} f_{514},\\
    f_{115}  & = & z f_{215} + z f_{310} f_{015} + z f_{314} f_{415} + z f_{315} f_{515},\\
    f_{120}  & = & z f_{220} + z f_{310} f_{020} + z f_{314} f_{420} + z f_{315} f_{520},\\
    f_{124}  & = & z f_{224} + z f_{310} f_{024} + z f_{314} f_{424} + z f_{315} f_{524},\\
    f_{125}  & = & z f_{225} + z f_{310} f_{025} + z f_{314} f_{425} + z f_{315} f_{525},
         \end{array}$
     
    $
 \begin{array}{lll}
    f_{200}  & = & z f_{510} f_{000} + z f_{514} f_{400},\\
    f_{204}  & = & z f_{510} f_{004} + z f_{514} f_{404},\\
    f_{205}  & = & z f_{510} f_{005} + z f_{514} f_{405},\\
    f_{210}  & = & z f_{510} f_{010} + z f_{514} f_{410} + z f_{515} f_{510},\\
    f_{214}  & = & z f_{510} f_{014} + z f_{514} f_{414} + z f_{515} f_{514},\\
    f_{215}  & = & z f_{510} f_{015} + z f_{514} f_{415} + z f_{515} f_{515},\\
    f_{220}  & = & z f_{510} f_{020} + z f_{514} f_{420} + z f_{515} f_{520},\\
    f_{224}  & = & z f_{510} f_{024} + z f_{514} f_{424} + z f_{515} f_{524},\\
    f_{225}  & = & z f_{510} f_{025} + z f_{514} f_{425} + z f_{515} f_{525},
         \end{array}$
     
    $
 \begin{array}{lll}
    f_{310}  & = & z + z f_{610},\\
    f_{314}  & = & z f_{614},\\
    f_{315}  & = & z f_{615},
         \end{array}$
     
    $
 \begin{array}{lll}
    f_{400}  & = & z f_{200} + z f_{310} f_{000} + z f_{314} f_{400},\\
    f_{404}  & = & z f_{204} + z f_{310} f_{004} + z f_{314} f_{404},\\
    f_{405}  & = & z f_{205} + z f_{310} f_{005} + z f_{314} f_{405},\\
    f_{410}  & = & z f_{210} + z f_{310} f_{010} + z f_{314} f_{410} + z f_{315} f_{510},\\
    f_{414}  & = & z + z f_{214} + z f_{310} f_{014} + z f_{314} f_{414} + z f_{315} f_{514},\\
    f_{415}  & = & z f_{215} + z f_{310} f_{015} + z f_{314} f_{415} + z f_{315} f_{515},\\
    f_{420}  & = & z f_{220} + z f_{310} f_{020} + z f_{314} f_{420} + z f_{315} f_{520},\\
    f_{424}  & = & z + z f_{224} + z f_{310} f_{024} + z f_{314} f_{424} + z f_{315} f_{524},\\
    f_{425}  & = & z f_{225} + z f_{310} f_{025} + z f_{314} f_{425} + z f_{315} f_{525},
         \end{array}$
     
    $
 \begin{array}{lll}
    f_{510}  & = & z f_{810} + z f_{710} f_{010} + z f_{714} f_{410} + z f_{715} f_{510},\\
    f_{514}  & = & z + z f_{814} + z f_{710} f_{014} + z f_{714} f_{414} + z f_{715} f_{514},\\
    f_{515}  & = & z f_{815} + z f_{710} f_{015} + z f_{714} f_{415} + z f_{715} f_{515},\\
    f_{520}  & = & z f_{820} + z f_{710} f_{020} + z f_{714} f_{420} + z f_{715} f_{520},\\
    f_{524}  & = & z f_{824} + z f_{710} f_{024} + z f_{714} f_{424} + z f_{715} f_{524},\\
    f_{525}  & = & z + z f_{825} + z f_{710} f_{025} + z f_{714} f_{425} + z f_{715} f_{525},
         \end{array}$
     
    $
 \begin{array}{lll}
    f_{610}  & = & z f_{810} + z f_{710} f_{010} + z f_{714} f_{410} + z f_{715} f_{510},\\
    f_{614}  & = & z f_{814} + z f_{710} f_{014} + z f_{714} f_{414} + z f_{715} f_{514},\\
    f_{615}  & = & z f_{815} + z f_{710} f_{015} + z f_{714} f_{415} + z f_{715} f_{515},
             \end{array}$
     
    $
 \begin{array}{lll}
    f_{710}  & = & z f_{610},\\
    f_{714}  & = & z f_{614},\\
    f_{715}  & = & z f_{615},
             \end{array}$
     
    $
 \begin{array}{lll}
    f_{810}  & = & z f_{520} f_{010} + z f_{524} f_{410} + z f_{525} f_{510},\\
    f_{814}  & = & z f_{520} f_{014} + z f_{524} f_{414} + z f_{525} f_{514},\\
    f_{815}  & = & z f_{520} f_{015} + z f_{524} f_{415} + z f_{525} f_{515},\\
    f_{820}  & = & z f_{520} f_{020} + z f_{524} f_{420} + z f_{525} f_{520},\\
    f_{824}  & = & z f_{520} f_{024} + z f_{524} f_{424} + z f_{525} f_{524},\\
    f_{825}  & = & z f_{520} f_{025} + z f_{524} f_{425} + z f_{525} f_{525}.\end{array}$

\bigskip

Using Maple (version 14) we can solve to obtain an expression for the algebraic 
generating function $f_{000}(z)$, which counts the number of words in $\lan\$$ 
of each length. Since words in $\lan\$$ of length $3n+1$ are in bijection with 
permutations in $\mathcal P$ of length $n$, the generating function $\sum_{n\geq 
0} c_nt^n$ where $c_n$ is the number of permutations of length $n$ in $\mathcal 
P$ is obtained by dividing $f_{000}$ by $z$ and substituting $z^3=t$.
\end{proof}

From the expression for the generating function we can easily obtain the first few terms of the sequence:

$1+z+2z^2+6z^3+24z^4+114z^5+592z^6+3216z^7+17904z^8+101198z^9+578208z^{10}+3332136z^{11}+19343408z^{12}+\dots .$

We can also use standard analytic combinatorial methods \cite{MR2483235} to 
deduce the asymptotic growth of the number of such permutations:
\begin{align*}
  c_n \sim \frac{\sqrt{25-11\sqrt{5}}}{2\sqrt{\pi n^3}} \cdot (2+2\sqrt{5})^n 
\cdot \left(1 + O(n^{-1}) \right).
\end{align*}



\end{document}